\documentclass[11pt]{amsart}
\usepackage{amssymb}
\usepackage{amsmath}
\usepackage{amsthm}
\usepackage{verbatim}
\usepackage[all]{xy}
\usepackage{url}
\usepackage{mathtools}
\usepackage{dsfont}
\usepackage{pdflscape}

\usepackage[mathscr]{euscript}
\usepackage{calligra}

\theoremstyle{plain}
\newtheorem{theorem}{Theorem}[section]

\newtheorem{question}[theorem]{Question}

\theoremstyle{definition}

\newtheorem{example}[theorem]{Example}

\newtheorem{subsec}[theorem]{}

\newcommand{\C}{{\mathds C}}
\newcommand{\R}{{\mathds R}}
\newcommand{\Q}{{\mathds Q}}

\newcommand{\ov}{\overline}

\newcommand{\into}{\hookrightarrow}

\newcommand{\labelt}[1]{\xrightarrow{\makebox[1.2em]{\scriptsize ${#1}$}}}

\newcommand{\longisoto}{{\ \labelt{\raisebox{-1.ex}{$\sim$}}\ }}

\newcommand{\isoto}{\longisoto}

\newcommand{\hs}{\kern 0.8pt}
\newcommand{\hssh}{\kern 1.2pt}
\newcommand{\hshs}{\kern 1.6pt}
\newcommand{\hssss}{\kern 2.0pt}

\newcommand{\hm}{\kern -0.8pt}
\newcommand{\hmm}{\kern -1.2pt}

\newcommand{\emm}{\bfseries}

\newcommand{\loc}{{\rm loc}}

\newcommand{\Gal}{{\rm Gal}}
\newcommand{\SU}{{\rm SU}}
\newcommand{\diag}{{\rm diag}}

\newcommand{\V}{\mathcal{V}}

\newcommand\Hon{{\rm H}^1\hm}
\newcommand{\Zl}{{\rm Z}}

\newcommand{\functor}{{\,\rightsquigarrow\,}}

\newcommand{\uptau}{{\hs^\tau\!}}

\newcommand{\Grp}{{\mathcal{G}rp}}
\newcommand{\PSet}{{\mathcal{PS}et}}
\newcommand{\Tori}{{\mathcal{T}\!\hm or}}
\newcommand{\Red}{{\mathcal{R}ed}}

\begin{document}

\title[Group structure on Galois cohomology]
{Is there a group structure on the Galois \\ cohomology of a reductive group \\ over a global field?}

\author{Mikhail Borovoi}
\address{%
Raymond and Beverly Sackler School of Mathematical Sciences,
Tel Aviv University,
6997801 Tel Aviv,
Israel}
\email{borovoi@tauex.tau.ac.il}

\thanks{This research was partially supported
by the Israel Science Foundation (grant 1030/22).}

\keywords{Galois cohomology, reductive groups, number fields,  global  fields}

\subjclass{
11E72
, 20G10
, 20G20
, 20G30
}

\begin{abstract}
Let $K$ be a global field, that is, a number field or a global function field.
It is known that the answer to the question in the title over $K$  is ``Yes" when $K$ has no real embeddings.
We show that otherwise the answer is ``No''.
Namely, we show that when $K$ is a number field admitting a real embedding,
 it is impossible to define a group structure on the first Galois cohomology sets $\Hon(K,G)$
 for all reductive $K$-groups $G$ in a functorial way.
\end{abstract}

\maketitle

\section{Introduction}
\label{s:intro}

Let $G$ be a reductive algebraic group defined over a field $K$.
Here and throughout this note we follow the convention of SGA3
where reductive groups are assumed to be connected.
We are interested in the first Galois cohomology set $\Hon(K, G)$;
for the definition see Serre \cite[Sections I.5 and III.1]{Serre}  or Section \ref{s:generalities} below.

When $G$ is abelian (a $K$-torus), the product of two 1-cocycles is a 1-cocycle,
and so $\Hon(K,G)$ has a natural structure of an abelian group.
We ask whether $\Hon(K,G)$ can be endowed with a functorial group structure
 for all reductive $K$ groups, not necessarily abelian.

We denote by $\Grp$ the category of groups,
and by $\PSet$ the category of pointed sets.
For a field $K$, we denote by $\Tori_K$ the category of $K$-tori,
and by $\Red_K$ the category of  reductive $K$-groups.

Consider the functor
\begin{equation}\label{e:Red-PSet}
\Red_K\functor \PSet,\quad\ G\mapsto \Hon(K,G)
\end{equation}
and its restriction to the category of $K$-tori
\begin{equation}\label{e:Tori-PSet}
\Tori_K\functor \PSet,\quad\ T\mapsto \Hon(K,T).
\end{equation}
Since a $K$-torus is an abelian $K$-group, we have a functor to the category of groups
\begin{equation}\label{e:Tori-Grp}
\Tori_K\functor \Grp,\quad\ T\mapsto \Hon(K,T)
\end{equation}
from which the functor \eqref{e:Tori-PSet} can be obtained by applying the forgetful functor.
We address the following question:

\begin{question}\label{q:group-structure}
For a given field $K$, is it possible to extend the functor \eqref{e:Tori-Grp} to a functor
\begin{equation*}
\Red_K\functor \Grp
\end{equation*}
from which \eqref{e:Red-PSet} can be obtained by applying the forgetful functor?
\end{question}

In other words, is it possible to construct a group structure
on the pointed set $\Hon(K,G)$ for each reductive $K$-group $G$,
with the following properties?
\begin{itemize}
\item  The unit element of the resulting
group is the neutral element of $\Hon(K,G)$;
\item for any homomorphism of reductive $K$-groups $\varphi\colon G\to G'$,
 the induced map $\varphi_*\colon \Hon(K,G)\to \Hon(K,G')$ is a group homomorphism;
\item for any $K$-torus $T$, the
new group structure coincides with the
standard group structure on $\Hon(K, T)$.
\end{itemize}
If Question~\ref{q:group-structure} has a positive answer for some field $K$, then
for any homomorphism $\psi\colon T\to G$ from a $K$-torus $T$
to a reductive $K$-group $G$, the induced map $\psi_*\colon \Hon(K,T)\to \Hon(K,G)$
is a group homomorphism with respect to the standard group structure on $\Hon(K,T)$
and the new group structure on $\Hon(K,G)$.

The answer to Question~\ref{q:group-structure} is trivially ``Yes'' when $K$ is a finite field,
because then by Lang's theorem \cite{lang-56} we have $\Hon(K,G)=1$ for all connected $K$-groups $G$.
However, one would expect the answer  ``No''  in most other cases,
because reductive groups in general are non-abelian,
and therefore the product of two cocycles may not be a cocycle.
Surprisingly, the answer to Question~\ref{q:group-structure} is ``Yes"
for any non-archimedean local field $K$;
see \cite[Corollary 5.4.1]{Borovoi-Memoir} in characteristic 0,
and \cite[Theorem 6.5.2(1)]{BK} in arbitrary characteristic.
Moreover,  the answer is also ``Yes'' for any  global  field $K$ without real places
(that is, a totally imaginary number field or a  global  function field);
see \cite[Theorem 5.11]{Borovoi-Memoir} in characteristic 0,
and \cite[Theorem 6.5.1]{BK} in arbitrary characteristic.
See also Labesse \cite[Section 1.6]{Labesse}, where a generalization to {\em quasi-connected} reductive groups is considered,
as well as Kottwitz~\cite[Proposition 6.4]{Kottwitz-Duke} and Gonz\'alez-Avil\'es~\cite[Theorem 5.8(1)]{GA},
where the corresponding abelian group structures are constructed in (different) non-functorial ways.

In this note we prove that when  $K$ is the field of real numbers  $\R$, or $K$ is a number field admitting a real embedding,
the answer to Question \ref{q:group-structure} is ``No''.
Our main result is the following theorem:

\begin{theorem}\label{t:no-gp-intro}
In the cases
\begin{enumerate}
\item [\rm (a)] when $K=\R$,  and
\item [\rm (b)] when $K$ is a number field admitting a real embedding,
\end{enumerate}
it is impossible to define a functorial in $G$ group structure
on the pointed sets $\Hon(K,G)$ for all reductive $K$-groups $G$ in such a way
that when $G=T$ is a $K$-torus, the new group structure on $\Hon(K,T)$ coincides with the standard one.
\end{theorem}

 The remainder of this note is structured as follows.
 In Section \ref{s:generalities} we recall the definition of the first Galois cohomology set.
 In Section \ref{s:R} we consider  Galois cohomology over $\R$ and prove Theorem \ref{t:no-gp-intro}(a).
 In Section \ref{s:number}  we consider  Galois cohomology over a number field and prove Theorem \ref{t:no-gp-intro}(b).

\section{Generalities on Galois cohomology}
\label{s:generalities}

\begin{subsec}
Let $K$ be a field. We denote by $\ov K$ an algebraic closure of $K$, and by $K^s$ the separable closure of $K$ in $\ov K$.
By $\Gal(K^s/K)$ we denote the absolute Galois group of $K$.

In this note, by an algebraic group we  always mean a {\em linear} algebraic group,
and by a $K$-group we  mean a linear algebraic group over $K$.
Let $G$ be a $K$-group, not necessarily commutative.
We work with the first Galois cohomology set
\[ \Hon(K,G)=\Zl^1(K,G)/\sim\]
see Serre \cite{Serre}.
Recall that $\Zl^1(K,G)$ is the set of 1-{\em cocycles}, that is, of  locally constant maps
\[c\colon \Gal(K^s/K)\to G(K^s)\]
satisfying the {\em cocycle condition}
\[c(\gamma_1,\gamma_2)=c(\gamma_1)\cdot {}^{\gamma_1}c(\gamma_2)\quad\ \text{for all}\ \,\gamma_1,\gamma_2\in\Gal(K^s/K).\]
Two cocycles  $c,c'$ are {\em equivalent} (we write $c\sim c'$)
if there exists $g\in G(K^s)$ such that
\[ c'(\gamma)=g^{-1}c(\gamma)\,^\gamma\! g
     \quad\ \text{for all}\ \, \gamma\in\Gal(K^s/K).\]
\end{subsec}

\begin{subsec}\label{ss:Func}
We describe the functoriality of $\Hon(K,G)$ in $G$.
Let $\varphi\colon G\to H$ be a homomorphism of $K$-groups.
The induced map
\[\varphi_*\colon \Hon(K,G)\to\Hon(K,H)\]
sends the class of a cocycle $c\in\Zl^1(K,G)$ to the class of $\varphi\circ c\in\Zl^1(K,H)$.
\end{subsec}

\section{Galois cohomology over $\R$}
\label{s:R}

For a (connected) reductive $\R$-group $G$, by abuse of terminology
we say that $G$ is {\em compact} if the Lie group $G(\R)$ is compact.
It is well known that $G$ is compact if and only if it is anisotropic.

\begin{theorem}\label{t:no-gp-R}
Let $G$ be a semisimple $\R$-group and $i\colon T\into G$ be a maximal compact torus.
Assume that $\Hon(\R,G)\neq 1$ and that $\#\Hon(\R,G)$ is odd.
Then $\Hon(\R,G)$ admits no group structure
compatible with the structure of pointed set and such that
\[i_*\colon \Hon(\R,T)\to\Hon(\R,G)\]
is a  group homomorphism.
\end{theorem}

\begin{proof}
Let $\xi\in\Hon(\R,G)$, $\xi\neq 1$.
The map
\[i_*\colon \Hon(\R,T)\to\Hon (\R,G)\]
is surjective by  \cite[Lemma 2.1(b) and Theorem 3.1]{Borovoi-CiM}.
Therefore, $\xi=i_*(\xi_T)$ for some $\xi_T\in\Hon(\R,T)$, $\xi_T\neq 1$.
We have $\xi_T^2=1$; see, for instance, \cite[Lemma 2.1(a)]{Borovoi-CiM}.
Thus $\xi_T$ is of order 2.
On the other hand, for any group structure on $\Hon(\R,G)$, since $\xi\neq 1$, the order of $\xi$
is not 1, and since $\#\Hon(\R,G)$ is odd, the order of $\xi$ is not 2.
Thus the image of the element $\xi_T$ of order 2 is not of order 1 or 2, and therefore
$i_*$ is not a group homomorphism.
\end{proof}

\begin{example}\label{ex-SU-R}
Consider the compact simply connected $\R$-group $G=\SU_4$, that is, the special unitary group  of the  Hermitian form
\begin{equation}\label{e:Hermitian}
 \mathcal{F}(x,y)=x_1\uptau\hm y_1+ x_2\uptau\hm y_2 + x_3\uptau\hm y_3+x_4\uptau\hm y_4
\end{equation}
where $\tau\colon \C\to \C$ denotes the complex conjugation in the field of complex numbers $\C$.
The pointed set  $\Hon(\R,\SU_4)$ classifies Hermitian  $4\times 4$ -matrices with determinant 1.
There are 3 equivalence classes of such matrices, with representatives
\[\diag(1,1,1,1),\quad \diag(-1,-1,1,1),\quad\diag(-1,-1,-1,-1).\]
Thus $\# \Hon(\R,G)=3$.
Let $i\colon T\into G$ be a compact maximal torus (say, the diagonal torus).
By Theorem \ref{t:no-gp-R}, the pointed set $\Hon(\R,G)$ has no groups structure
such that $i_*$ is a group homomorphism.
\end{example}

Theorem \ref{t:no-gp-R} together with Example \ref{ex-SU-R} prove Theorem \ref{t:no-gp-intro}(a).

\section{Galois cohomology over a number field}
\label{s:number}

Let $K$ be a number field, that is, a finite extension of the field of rational numbers $\Q$.
Let $\V_\R(K)$ denote the set of real places of $K$.
We identify a real place $v\in \V_\R(K)$ with the corresponding  embedding $\phi_v\colon K\into\R$.
If $G$ is a $K$-group, we write $G_v\coloneqq G\times_{K,\phi_v}\! \R$ for the corresponding $\R$-group.

\begin{theorem}\label{t:no-gp-K}
Let $K$ be a number field, and suppose that $\V_\R(K)$ is not empty.
Let $G$ be a {\emm simply connected} semisimple $K$-group
such that  the cardinality $\#\Hon(\R, G_v)$ is odd for all $v\in\V_\R(K)$,
and that for some $v_0\in \V_\R(K)$  we have $\Hon(\R,G_{v_0})\neq 1$.
Let $i\colon T\into G$ be a torus that splits over a quadratic extension $L/K$
and such that $i_{v_0}\colon T_{v_0}\into G_{v_0}$ is a maximal compact torus of $G_{v_0}$.
Then $\Hon(K,G)$ admits no group structure
compatible with the structure of pointed set and such that
\[i_*\colon\Hon(K,T)\to\Hon(K,G)\]
is a group homomorphism.
\end{theorem}

\begin{proof}
Consider the commutative diagram
\[
\xymatrix@C=13mm{
\Hon(K,T)\ar[r]^-{i_*}\ar[d]_-{\loc_{v_0}}  &\Hon(K,G)\ar[d]^-{\loc_{v_0}}\\
\Hon(K_{v_0},T)\ar[r]^-{i_{v_0,*}}            &\Hon(K_{v_0},G)
}
\]
where $K_{v_0}\cong\R$. In this diagram, the vertical arrows are surjective
by \cite[Proposition 6.17 on p.~337]{PR}.
The bottom horizontal arrow is surjective;
see the reference in the proof of Theorem \ref{t:no-gp-R}.
Thus the composite map $\Hon(K,T)\to \Hon(K_{v_0},G)$ in the diagram is surjective.

By assumption there exists a non-trivial element $\xi_{v_0}\in \Hon(K_{v_0},G)$.
Then $\xi_{v_0}=\loc_{v_0}(\xi)$ for some  $\xi\in \Hon(K,G)$
such that  $\xi=i_*(\xi_T)$ for some $\xi_T\in\Hon(K,T)$.
Since $\xi_{v_0}\neq 1$, we see that $\xi\neq 1$ and $\xi_T\neq 1$.

Since the $K$-torus $T$ splits over the extension $L/K$,
by \cite[Lemma (1.9.2)]{Sansuc} we have $\Hon(K,T)=\Hon(L/K,G)$.
Since $L/K$ is a quadratic extension, by \cite[Chapter IV, Section 6, Corollary 1 of Proposition 8]{CF}
we have $\xi_T^2=1$.
Since  $\xi_T\neq 1$, we see that  the order of $\xi_T$ is 2.
On the other hand, since $G$ is simply connected, by the Hasse principle for $G$
(Kneser-Harder-Chernousov, see \cite[Theorem 6.6 on p.~286]{PR}\hs),
there is a bijection
\[\Hon(K,G)\isoto\!\!\! \prod_{v\in \V_\R(K)}\!\! \Hon(K_v,G),\]
whence $\Hon(K,G)$ is a finite set of cardinality
\[\#\Hon(K,G)\, =\!\!\prod_{v\in\V_\R(K)}\!\!\#\Hon(K_v,G).\]
Since by assumption  $\#\Hon(K_v,G)$ is odd for all $v\in \V_\R(K)$,
we see that $\#\Hon(K,G)$ is odd.
For any group structure on $\Hon(K,G)$, the element $\xi$ is not of order 1 because $\xi\neq 1$,
and not of order 2 because $\#\Hon(K,G)$ is odd.
Thus $\xi_T\in \Hon(K,T)$ is of order 2, but $\xi=i_*(\xi_T)\in \Hon(K,G)$ is not of order 1 or 2.
We conclude that $i_*$ is not a group homomorphism.
\end{proof}

\begin{example}\label{ex:number-SU4}
Let $K$ be a number field such that $\V_\R(K)$ is not empty.
Let $a\in K^\times$ be an element such that $\phi_v(a)\in\R$
is positive for all $v\in\V_\R(K)$, for example, $a=1$.
Set $L=K(\sqrt{-a}\hs)$.
Let $G=\SU_{4,L/K}$\hs, the special unitary group
of the Hermitian form \eqref{e:Hermitian} on the 4-dimensional vector space
$L^4$ where $\tau\colon L\to L$ is the nontrivial element of the Galois group $\Gal(L/K)$ of order 2.
Then for each $v\in \V_\R(K)$, the  $\R$-group $G_v$ is a compact $\R$-group isomorphic to $\SU_4$,
and therefore $\#\Hon(K_v,G)=\#\Hon(\R, G_v)=3$; see Example \ref{ex-SU-R}.
Let $i\colon T\into G$ be the diagonal torus in $G=\SU_{4,L/K}$.
Then $T$ splits over the quadratic extension $L/K$, and for any $v\in \V_\R(K)$
the corresponding torus $T_v$ is a compact maximal torus of $G_v$.
By Theorem \ref{t:no-gp-K}, there is no group structure on $\Hon(K,G)$
such that $i_*$ is a group homomorphism.
\end{example}

Theorem \ref{t:no-gp-K} together with Example \ref{ex:number-SU4} prove Theorem \ref{t:no-gp-intro}(b).

\section*{Acknowledgements}
The author thanks Zinovy Reichstein for most helpful email correspondence.
This note was partially written during the author’s visit
to Max Planck Institute for Mathematics, Bonn,
and he thanks this institute for its hospitality, support, and
excellent working conditions.

\end{document}